\documentclass[12pt]{amsart}
\usepackage{amsmath,amsthm}
\usepackage{amsfonts}
\usepackage{amssymb}

\usepackage{graphicx,xy,pstricks}
\input xy
\xyoption{all}

\newcommand{\executeiffilenewer}[3]{%
 \ifnum\pdfstrcmp{\pdffilemoddate{#1}}%
 {\pdffilemoddate{#2}}>0%
 {\immediate\write18{#3}}\fi%
}

\newtheorem{lemma}{Lemma}[section]
\newtheorem{proposition}[lemma]{Proposition}

\newtheorem{question}[lemma]{Question}

\theoremstyle{definition}

\newtheorem{remark}[lemma]{Remark}

\newcommand{\R}{\mathbb{R}}
\newcommand{\Z}{\mathbb{Z}}
\newcommand{\C}{\mathbb{C}}

\newcommand{\sldeux}{\mathrm{SL}_2}
\newcommand{\liedeux}{\mathrm{sl}_2}
\newcommand{\hatcal}[1]{\widehat{\mathcal{#1}}}

\DeclareMathOperator{\tr}{Tr}
\DeclareMathOperator{\homo}{Hom}
\DeclareMathOperator{\cs}{CS}
\DeclareMathOperator{\ad}{Ad}
\DeclareMathOperator{\hol}{Hol}
\DeclareMathOperator{\diverg}{div}
\DeclareMathOperator{\bil}{Bil}

\DeclareMathOperator{\lk}{Lk}
\DeclareMathOperator{\id}{Id}

\title{The Kauffman skein module at first order}
\date{}
\author{Julien March\'e}
\address{Institut de mathématiques de Jussieu - Paris Rive Gauche\\
Universit\'{e} Pierre et Marie Curie\\
75252 Paris c\'{e}dex 05\\
France}
\email{julien.marche@imj-prg.fr}
\begin{document}
\maketitle
\begin{abstract}
For a 3-manifold $M$ with boundary, we study the Kauffman module with indeterminate equal to $-1+\epsilon$ where $\epsilon^2=0$. We conjecture an explicit relation between this module and the Reidemeister torsion of $M$ which we prove in particular cases. As a maybe useful tool, we then introduce a notion of twisted self-linking and prove that it satisfies the Kauffman relations at first order. These questions come from considerations on asymptotics of quantum invariants.
\end{abstract}
\section{Introduction}
Given a 3-manifold $M$, a ring $R$ and an invertible element $t$ in $R$, the Kauffman bracket skein module is the free $R$-module generated by isotopy classes of banded links in $M$ modulo local relations depending on $t$, see \cite{hp}. When $R=\C$ and $t=-1$, the disjoint union operation makes it into the algebra of regular functions on the character variety of $\pi_1(M)$ into $\sldeux(\C)$. When $t$ is a root of unity (say of order $2k$), this module can be used to construct the Topological Quantum Field Theory with level $k$ and gauge group SU$_2$. The purpose of this article is to study specifically the case when $R=\C[\epsilon]/(\epsilon^2)$ and $t=-1+\epsilon$: we call it the {\it derived skein module} of $M$. If $M=\Sigma\times [0,1]$, the derived Kauffman module is again an algebra, deforming the one with $\epsilon=0$:  the deformation corresponds to the Goldman Poisson bracket as already observed in \cite{tur2}. 
If $M$ is a 3-manifold with boundary, the derived Kauffman module is a module over the derived Kauffman algebra of the neighborhood of the boundary. We show in this article that this structure is described by a second order differential operator on the character variety of $M$ (more precisely, on its restriction to the character variety of the boundary of $M$). We conjecture that this operator $\Upsilon$ is given by 
\begin{equation}\label{transport1}
\Upsilon(f)=-\frac{1}{2}\frac{L_{X_f}T}{T}
\end{equation}
where $f$ is a regular function on the character variety of $\partial M$ vanishing on the restriction of characters of $M$, $X_f$ is its Hamiltonian vector field and $T$ is the (twisted) Reidemeister torsion. We may think of this operator as a kind of symplectic Laplacian. We prove this formula for a thickened surface and a handlebody. It can be also checked by hand for torus knots and for the figure eight knot, see \cite{laurent,lj2}.

This formula is a manifestation of a well-known conjectural fact: the asymptotics of quantum invariants of a 3-manifold $M$ when the level $k$ tends to infinity is governed at first order by the character variety of $M$, the Chern-Simons functional and the Reidemeister torsion. This fact was already known to E. Witten for the case of closed 3-manifolds: various generalizations of the volume conjecture for knots in $S^3$ predict the same kind of phenomena, see for instance \cite{gm,gl}. In \cite{lj1,lj2}, the Witten conjecture for the expansion of the Witten-Reshetikhin-Turaev invariants of the Dehn fillings of the figure eight knot was obtained by proving another version of the Witten conjecture for the figure eight knot complement. One of the key technical inputs were an explicit differential equation relating the Reidemeister torsion to the $q$-differential equation satisfied by the colored Jones polynomial. This equation was called the transport equation of the Reidemeister torsion, following a standard terminology in semi-classical analysis: it is equivalent to Equation \eqref{transport1} in the case of knot complements. 

The present article generalises this transport equation for any manifold with boundary, using the beautiful relation between skein modules and $q$-differential equations discovered in \cite{fgl}, see \cite{le} for a review. The first section starts with a short motivation from quantum topology, then we state the conjecture and explain its relation with the quantum $A$-polynomial and the colored Jones polynomial.
 In the second section, we prove the formula in the case of thickened surfaces and handlebodies, assuming some familiarity with the symplectic structure on character varieties of surfaces. The last section presents the notion of twisted self-linking: this takes the form of a map $\Phi$ from the derived skein module of $M$ to a space $\mathcal{H}(M,\rho)$ depending on a representation $\rho:\pi_1(M)\to\sldeux(\C)$. For a knot $K$ in $M$, the element $\Phi(K)$ may be thought of as the self-linking of $K$ with coefficients in the adjoint representation $\ad_\rho$ and at the same time as the Hessian at $[\rho]$ of the trace function $f_K$ on the character variety mapping $[\rho]$ to $\tr\rho(K)$. I believe that this notion should yield a proof of Equation \eqref{transport1}. 

{\bf Acknowledgements:} I would like to thank L. Charles, R. Detcherry, S. Garoufalidis, M. Heusener T. Q. T. L\^e and J. Porti for their kind interest relating these questions and both referees for useful suggestions. 

\section{Motivations}

\subsection{Witten expansion conjecture}\label{witten}
Let $M$ be a closed oriented 3-manifold satisfying $H_1(M,\Z)=0$. Suppose moreover that for any non-trivial (and hence irreducible) representation $\rho:\pi_1(M)\to\sldeux(\C)$ we have $H^1(M,\ad_\rho)=0$. This is equivalent to the assumption that the following character variety  is reduced and of dimension 0:
\[X(M)=\homo(\pi_1(M),\sldeux(\C))/\!\!/\sldeux(\C).\]

The Witten-Reshetikhin-Turaev invariant of $M$ with gauge group SU$_2$ and level $k$ is a topological invariant of $M$ denoted by $Z_k(M)$ which conjecturally satisfies the following expansion:

\[ Z_k(M)=\sum_{\rho} e^{ik\cs(M,\rho)+\frac{i\pi}{4}I(M,\rho)}\sqrt{T(M,\rho)}+O(k^{-3/2})\]
In this formula, $\rho$ runs over conjugacy classes of non trivial morphisms from  $\pi_1(M)$ to SU$_2$ and $\cs(M,\rho)\in \R/2\pi\Z$, $I(M,\rho)\in \Z/8\Z$ and $T(M,\rho)\in (0,\infty)$ are topological invariants called respectively the Chern-Simons invariant, the spectral flow and the Reidemeister torsion.

The WRT invariant can be extended to pairs $(M,L)$ where $L$ is a banded link in $M$, that is an oriented submanifold diffeomorphic to a disjoint union of annuli, see for instance \cite{bhmv1}. As an invariant of $L$, it satisfies the celebrated Kauffman relations, that is if $L_{\times}, L_{\infty},L_0$ differ as in  Figure \ref{kauf}, then the invariants should satisfy the following equations where $\zeta_k=-e^{\frac{i\pi}{2k}}$: 
\[Z_k(M,L_\times)=\zeta_kZ_k(M,L_\infty)+\zeta_k^{-1}Z_k(M,L_0).\]

 \begin{figure}[htbp]
\centering
  \def\svgwidth{8cm}
 \executeiffilenewer{kaufman.svg}{kaufman.pdf}%
 {inkscape -z -D --file=kaufman.svg %
 --export-pdf=kaufman.pdf --export-latex}%
 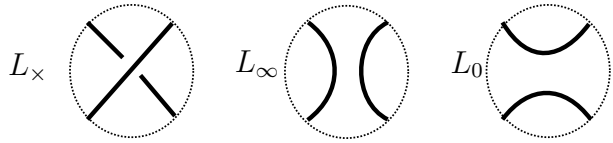%

  \caption{Kauffman triple}
  \label{kauf} 
\end{figure}
Moreover, if $L\cup O$ is the union of $L$ with a trivial banded knot, then one should have $Z_k(M,L\cup O)=(-\zeta_k^2-\zeta_k^{-2})Z_k(M,L)$. This invariant was introduced by Witten using a path integral that we describe below. 

Let $\alpha\in \Omega^1(M,\rm{su}_2)$ be a 1-form on $M$ with values in the Lie algebra su$_2$. We may think of it as a connection on a trivial SU$_2$ principal bundle and introduce the Chern-Simons action of $\alpha$ as the integral 
\[\cs(\alpha)=-\frac{1}{4\pi}\int_M\tr(\alpha\wedge d\alpha+\frac{2}{3}\alpha\wedge[\alpha\wedge\alpha]).\]
 Denoting by $L_i$  the components of $L$ and by $\hol_{L_i}\alpha$ the holonomy of $\alpha$ along $L_i$, one can write the original  path integral:
\[ Z_k(M,L)=\int e^{ik\cs(\alpha)}\prod_i(-\tr \hol_{L_i} \alpha) \mathcal{D}\alpha.\]

The Witten conjecture comes from a formal stationary phase expansion in this path integral, observing that critical points of the Chern-Simons action are flat connections (i.e. satisfy $d\alpha+\frac 1 2 [\alpha\wedge\alpha]=0$ ) and gauge equivalence class of flat connections correspond to conjugacy class of representations in SU$_2$. 

Let $Z_k(M,\rho)=e^{ik\cs(M,\rho)+\frac{i\pi}{4}I(M,\rho)}\sqrt{T(M,\rho)}$ be the term corresponding to $\rho$ in the Witten expansion of $Z_k(M)$. 
If we add the link $L$, the term $Z_k(M,\rho)$ should be at 0-th order multiplied by the value $\chi_\rho(L)=\prod_i (-\tr \rho(L_i))$. Let us denote by $\Phi_\rho(L)$ the first order term, that is such that the following expansion holds:
\[ Z_k(M,L)=\sum_{\rho} Z_k(M,\rho)\left(\chi_\rho(L)+\frac{i\pi}{2k}\Phi_\rho(L)\right)+O(k^{-3/2})\]

Factorizing $Z_k(M,\rho)$, we guess that the Kauffman relations should yield:
\begin{equation}\label{phi} \Phi_\rho(L_\times)+\Phi_\rho(L_\infty)+\Phi_\rho(L_0)=-\chi_\rho(L_\infty)+\chi_\rho(L_0)\text{ and }\Phi_\rho(L\cup O)=-2\Phi_\rho(L).
\end{equation}

One of the goals of this article is to construct rigorously such an invariant. Heuristically, it should be defined as a contraction of the Hessian of the function $\alpha\mapsto \tr \hol_\gamma \alpha$ with the inverse Hessian of the Chern-Simon action. We did not try to make sense of such a contraction although it was the source of inspiration of the more formal construction developed in Section \ref{selflink}. 

\subsection{Differential geometry on the character variety}

Let $R$ be a ring and $t$ be an invertible element in $R$. The Kauffman skein module $\mathcal{S}(M,R,t)$ of a compact connected 3-manifold $M$ (maybe with boundary) is defined as the free $R$-module generated by isotopy classes of banded links in $M$ (including the empty link) modulo the Kauffman relations $[L_\times]-t[L_\infty]-t^{-1}[L_0]$ and $[L\cup O]+(t^2+t^{-2})[L]$. 

If $R=\C$ and $t=-1$, the disjoint union makes $\mathcal{S}(M,\C,-1)$ into an algebra, isomorphic to the algebra of regular functions on the character variety $X(M)$, see \cite{numazu} for a review. We will be concerned in this article with the {\em derived skein module} which consists in setting $R=\C[\epsilon]/(\epsilon^2)$ and $t=-1+\epsilon$. For short, we will denote the first algebra by $\mathcal{S}(M)$ and the derived skein module by $\mathcal{S}'(M)$. 

Denoting by $E\in\operatorname{End}(\mathcal{S}'(M))$ the multiplication by $\epsilon$ we have $\mathcal{S}(M)=\mathcal{S}'(M)/\operatorname{Im}(E)$. As $E$ satisfies $E^2=0$, we have a complex $0\to \mathcal{S}'(M)/\operatorname{Im}(E)\overset{E}{\to}\mathcal{S}'(M)\to\mathcal{S}(M)\to 0$ which may fail to be exact only at the first entry. We will say that $M$ is {\it regular} if the map $E:\mathcal{S}(M)\to \mathcal{S}'(M)$ is injective. 


If $M$ is diffeomorphic to a thickened surface of the form $\Sigma\times[0,1]$, then stacking a banded link over the other makes $\mathcal{S}(M,R,t)$ into an algebra, commutative if $t^2=1$. Moreover, projecting a link on $\Sigma$ and recursively solving the crossings show that the {\em multicurves} of $\Sigma$ - that is embedded 1-manifold without components bounding a disc - form a basis of $\mathcal{S}(M,R,t)$. 
This gives in particular a natural isomorphism $\mathcal{S}'(M)\simeq \mathcal{S}(M)\otimes \C[\epsilon]/(\epsilon^2)$, showing that $M$ is regular in that case. However, one should be careful that the preceding isomorphism is not an isomorphism of algebras as one has for $f+\epsilon f'$ and $g+\epsilon g'$ in $\mathcal{S}(M)$ the following identity: 
\begin{equation}\label{multiplication}
(f+\epsilon f')\cdot (g+\epsilon g')=fg+\epsilon(fg'+f'g+\frac{1}{2}\{f,g\})
\end{equation}
where $\{f,g\}$ is the Goldman Poisson bracket, see \cite{bfk}.

This formula shows that the derived skein module of a thickened surface $\Sigma\times [0,1]$ reflects the Poisson structure of the character variety $X(\Sigma)$, so that one is naturally leaded to ask whether the derived skein module of a manifold $M$ with boundary $\Sigma$ contains some geometric structure related to the character variety $X(M)$. Having very little structure, the derived skein module $\mathcal{S}'(M)$ is by itself no much information, however, we will use the relation between the derived skein module of $M$ and the one of a neighborhood of its boundary.

Let $M$ be a 3-manifold with boundary $\Sigma$ and denote by $BM$ a tubular neighborhood of $\Sigma$, hence homeomorphic to a thickened surface. 
The inclusion map $i:BM\to M$ induces a map $i_*: \mathcal{S}(BM)\to \mathcal{S}(M)$. Its kernel is an ideal $\mathcal{P}$ corresponding to a closed submanifold $Y\subset X(\Sigma)$. Standard arguments of algebraic geometry show that $Y$ is the Zariski closure of $r(X(M))\subset X(\Sigma)$ where $r$ is the restriction map induced by $i_*$. For instance if $M$ has toric boundary, $\mathcal{P}$ is generated by the celebrated $A$-polynomial. Moreover, Poincar\'e duality implies that $Y$ is a Lagrangian submanifold of $X(\Sigma)$. In algebraic terms, it can be simply written $\{\mathcal{P},\mathcal{P}\}\subset \mathcal{P}$, we will say in that case that $\mathcal{P}$ is a {\em Lagrangian ideal}.

Similarly, define $\mathcal{P}'$ to be the kernel of the map $i_*:\mathcal{S}'(BM)\to \mathcal{S}'(M)$. This fits into the following diagram where rows are exact.
\[\xymatrix{
0\ar[r] & \mathcal{P}\ar[r]\ar[d] & \mathcal{S}(BM)\ar[r]\ar[d] & \mathcal{S}(M)\ar[d] \\
0\ar[r] & \mathcal{P}'\ar[r]\ar[d] & \mathcal{S}'(BM)\ar[r]\ar[d] & \mathcal{S}'(M)\ar[d] \\
0\ar[r] & \mathcal{P}\ar[r] & \mathcal{S}(BM)\ar[r] & \mathcal{S}(M)}\]

Let us show in the following lemma that regularity assumptions (exactness) on the derived skein modules imply that $Y$ is Lagrangian, giving a purely skein theoretic ``proof" of this important fact. 
\begin{lemma}
If the complex $0\to \mathcal{P}\to \mathcal{P}'\to \mathcal{P}\to 0$ is exact (for short, if $\mathcal{P}$ is regular) then $\mathcal{P}$ is a Lagrangian ideal.
\end{lemma}
\begin{proof}
Let $f,g$ be elements of $\mathcal{P}$, by exactness, there exists $f',g'\in S(BM)$ such that $f+\epsilon f', g+\epsilon g'\in \mathcal{P}'$. As $\mathcal{P}'$ is an ideal, the commutator $[f+\epsilon f',g+\epsilon g']=\epsilon\{f,g\}$ belongs to $\mathcal{P}'$. Being a multiple of $\epsilon$, by the exactness assumption again, we derive $\{f,g\}\in \mathcal{P}$. 
\end{proof}

\begin{lemma}
Suppose that $\mathcal{P}$ is regular and define a map $\Upsilon: \mathcal{P}\to S(BM)/\mathcal{P}$ in the following way: 
for $f\in \mathcal{P}$, choose $f'\in S(BM)$ such that $f+\epsilon f'\in \mathcal{P}'$ and set $\Upsilon(f)=f'$. 
This operator is well-defined and satisfies for any $f\in\mathcal{P}$ and $g\in S(BM)$ the equality
\[ \Upsilon(gf)= g\Upsilon(f)+\frac{1}{2}\{g,f\}.\]
\end{lemma}

\begin{proof}
This follows immediately from the fact that for any $g+\epsilon g'\in S'(BM)$ and $f+\epsilon f'\in\mathcal{P}'$ one has $(g+\epsilon g')(f+\epsilon f')\in \mathcal{P}'$ and then $\Upsilon(gf)=g'f+gf'+\frac{1}{2}\{g,f\}=g\Upsilon(f)+\frac{1}{2}\{g,f\}\mod \mathcal{P}$. 
\end{proof}

We would like to understand this operator in terms of the differential geometry of $X(M)$ and $X(\Sigma)$. Let $X_f$ be the Hamiltonian vector field on $Y$ associated to $f\in \mathcal{P}$. Formally, $X_f$ is the derivation of $S(BM)/\mathcal{P}$ defined by $X_f(g)=\{f,g\}$. Let $T$ be the Reidemeister torsion of $M$ viewed as a rational volume form on $X(M)$ (see \cite{porti} or \cite{numazu} Section 4 for formal a construction of this form). 

\begin{question}\label{conjtransport}
Show that for any $f\in \mathcal{P}$, we have the following equality: 
\begin{equation}\label{transport}-2\Upsilon(f)=\frac{L_{X_f} T}{T}=\diverg_T(X_f).\end{equation} 
In words, elements of the ideal $\mathcal{P}'$ are obtained by substracting to $f\in\mathcal{P}$ twice the divergence of the Hamiltonian vector field of $f$ with respect to the Reidemeister torsion.
\end{question}
Let us rewrite this equation in local coordinates for concreteness: one can suppose that $(x_1,\ldots, y_n)$ are local Darboux coordinates on $X(\Sigma)$ (i.e. $\omega=\sum_i dx_i\wedge d y_i$) such that $Y$ is given by $y_1=\ldots=y_n=0$. 
Then, the Reidemeister torsion has the form $T=a(x_1,\ldots,x_n)dx_1\wedge\cdots\wedge dx_n$ and one computes for $f$ vanishing on $Y$:
\[\frac{L_{X_f}T}{T}=\sum_i \frac{\partial^2f}{\partial x_i \partial y_i}+ \frac{\{f,a\}}{a}\]


We compute that, modulo $\mathcal{P}$ one has $\diverg_T(X_{fg})=\diverg_T(gX_f)=g\diverg_T(X_f)+X_f(g)=g\diverg_T(X_f)+\{f,g\}$. Hence the difference $D(f)=\Upsilon(f)+\frac{1}{2}L_{X_f}T/T$ is a derivation of $\mathcal{S}(BM)/\mathcal{P}$ - or vector field on $Y$ - that vanishes conjecturally.

\subsection{The quantum $A$-polynomial at first order}

Let $R$ be the ring $\C[t,t^{-1}]$. We set $\widehat{\mathcal{S}}(M)=\mathcal{S}(M,R,t)$. If $N$ is a 3-manifold and $BN$ is a tubular neighborhood of its boundary, we will denote by $\hatcal{P}$ the kernel of the natural map $\hatcal{S}(BN)\to \hatcal{S}(N)$. In the case when $N$ is the complement of a knot in $S^3$, this ideal has an important interpretation in terms of the colored polynomials of $K$ that we recall quickly below.

Indeed in that case, $\hatcal{S}(BN)$ is the sub-algebra of the quantum torus 
\[\mathcal{T}=\Z[t^{\pm 1}]\langle L^{\pm 1},M^{\pm 1}\rangle/(LM-t^2ML)\]
 invariant by the involution $\sigma(L^aM^b)=L^{-a}M^{-b}$. The meridian of the knot corresponds to $-M-M^{-1}$ whereas the longitude corresponds to $-L-L^{-1}$. 
 
 The element $L$ acts on a sequence $f:\Z\to R$ by the formula $(Lf)_n=f_{n+1}$ while $M$ acts by $(Mf)_n=t^{2n}f_n$. An easy computation shows that elements in
$\hatcal{P}$ annihilate the sequence $(J^K_n)_{n\in\Z}$ of colored Jones polynomials. 
Moreover, putting $t=-1$, an element of $\hatcal{P}$ gives an element of $\mathcal{P}$, a simple fact which fully motivates the celebrated AJ-conjecture, see \cite{le}.

Before going further, define $\sigma_0,\sigma_1:\mathcal{T}\to\C[L^{\pm 1},M^{\pm 1}]$ two maps satisfying the following relations:
\begin{align*}
\sigma_0(P(t,L,M))=P(-1,L,M), \quad \sigma_1(M)=\sigma_1(L)=0,\\
\sigma_1(PQ)=\sigma_0(P)\sigma_1(Q)+\sigma_1(P)\sigma_0(Q)+\frac{1}{2}\{\sigma_0(P),\sigma_0(Q)\}.
\end{align*}
where the Poisson bracket in the torus is given by the symplectic form $\omega=\frac{dL\wedge dM}{LM}$. 
We extend the map $\sigma_0+\epsilon \sigma_1:\mathcal{T}\to \C[L^{\pm 1},M^{\pm 1}][\epsilon]/(\epsilon^2)$ by linearity sending $t$ to $-1-\epsilon$.  
These equations are reminiscent from Equation \eqref{multiplication} and immediately imply that if $f+\epsilon f'$ is the reduction modulo $\epsilon^2$ of a polynomial $P\in \hatcal{S}(BM)\subset\mathcal{T}$, then $\sigma_0(P)=f$ and $\sigma_1(P)=f'$. By analogy with the theory of Toeplitz operators, let us call them principal and subprincipal symbols. 

\subsubsection{The abelian case}

Consider again the case where $N$ is the complement of a knot in $S^3$. Then the character variety $X(N)$ contains a component of abelian characters indexed by $\C^*$ where $z\in\C^*$ corresponds to a representation $\rho_z$ mapping the meridian to a diagonal matrix with entries $e^z,e^{-z}$. In that case, the Reidemeister torsion form is $\frac{e^z-e^{-z}}{\Delta_N(e^{2z})}dz$ where $\Delta_N$ is the Alexander polynomial of $N$. 

It was shown in \cite{gl} as a strengthening of \cite{glvieux} that one has $$\lim\limits_{n\to \infty}J_n^K(-e^{z/2n})=\frac{e^z-e^{-z}}{\Delta(e^{2z})}$$
 uniformly for $z$ in a neighbourhood of $0$. Be careful that we use the normalisation of the Jones polynomial which gives for the trivial knot $J^O_n=t^{2n}-t^{-2n}$. 
 
 If $f$ is holomorphic in a neighbourhood of $0$ and $f_n\in \C[z]$ is a sequence satisfying $f_n(-e^{z/2n})= f(z)+o(\frac{1}{n})$ then a simple computation shows that 
\begin{enumerate}
\item $(t f)_n(-e^{z/2n})= -f(z)-\frac{z}{2n}f(z)+o(\frac{1}{n})$
\item $(Mf)_n(-e^{z/2n})= e^{z}f(z)+o(\frac{1}{n})$
\item $(Lf)_n(-e^{z/2n})= f(z)+\frac{z}{n}f'(z)+o(\frac{1}{n})$
\end{enumerate}
These formulas may be written for $P=M$ or $P=L$ in the following way:
$$(Pf)_n(-e^{z/2n})= \sigma_0(P)(e^{z},1)f(z)+\frac{z}{2n}\left(\sigma_1(P)-2X_{\sigma_0(P)}(e^{z},1)\right)f(z)+o(\frac{1}{n})$$
Being compatible with composition, this formula holds for any $P\in \mathcal{T}$. We get that if $P\in\mathcal{T}$ satisfies $P J^K=0$ then 
$\sigma_0(P)(e^{z},1)=0$ and $(\sigma_1(P)(e^z,1)-2X_{\sigma_0(P)}(e^z,1))(\frac{e^z-e^{-z}}{\Delta(e^{2z})})=0$.

Reinterpreting $\frac{e^z-e^{-z}}{\Delta(e^{2z})}$ as the Reidemeister torsion, this proves that the formula \eqref{conjtransport} holds in that case where the submanifold $Y$ is defined by the equation $L=1$.

\subsubsection{The figure eight knot}
In \cite{lj1,lj2}, the colored Jones polynomials were packaged into a sequence of holomorphic sections over the standard torus $T^2$,  $\Psi_k\in H^0(T^2,L^k\otimes \delta)$ where $L$ is a pre-quantum bundle and $\delta$ a half-form bundle. This sequence were called the {\em knot state} and were the main ingredient for the proof of the Witten expansion conjecture for the manifolds obtained by Dehn fillings along the figure-eight knot. In that case, it was shown in \cite{lj2} that the state $\Psi_k$ is a Lagrangian state, concentrating on $Y$, with phase the Chern-Simons action and amplitude the Reidemeister torsion. It followed from an interpretation of the equation $P(t,L,M) J^K=0$ as a Toeplitz operator annihilating the knot state $\Psi_k$. The Reidemeister torsion appeared through a differential equation involving the symbol $\sigma_0(T)$ and the subprincipal symbol $\sigma_1(T)$ of the Toeplitz operator $T$. 


It happens that these symbols are (up to a factor of $i$) the same as the ones defined above. Hence, substituting $f$ and $f'$ with $\sigma_0$ and $\sigma_1$ shows immediately that the transport equation of Proposition 5.3 in \cite{lj2} is equivalent to Equation \eqref{transport}. 

\section{Examples}
\begin{proposition}
Equation \eqref{transport} holds in the case where $M=\Sigma\times[0,1]$ where $\Sigma$ is a closed surface.
\end{proposition}
In that case, $BM$ is homeomorphic to $M\amalg M$ hence $\mathcal{S}'(BM)=\mathcal{S}'(M)\otimes \mathcal{S}'(M)$ and the ideal $\mathcal{P}'$ is clearly generated by elements of the form $\gamma\otimes 1-1\otimes\gamma$ for any essential simple curve $\gamma\subset \Sigma$. The restriction map is simply the diagonal embedding $X(M)=X(\Sigma)\to X(\Sigma)^2$ and the corresponding element of $\mathcal{P}'$ reads $F=f\circ p_1-f\circ p_2$ where $p_1$ and $p_2$ are the two projections and $f$ is the trace function corresponding to $\gamma$ (notice that there are no terms proportional to $\epsilon$). The Reidemeister torsion on $M$ is the Liouville form $\mu$ associated to the symplectic form on $X(\Sigma)$, and the Hamiltonian vector field associated to $F$ restricts to the diagonal as $X_f$. Hence the transport equation amounts to $L_{X_f}\mu=0$ which is obvious from the fact that any Hamiltonian flow preserves the symplectic and Liouville forms. 

\begin{proposition}\label{handle}
Equation \eqref{transport} holds in the case where $M$ is a handlebody.
\end{proposition}
\begin{proof}
Let $(\gamma_i)_{i\in I}$ be the collection of all essential curves in $\Sigma$ bounding discs in $M$. Then the ideal $\mathcal{P}'$ is generated by elements of the form $[\gamma\#\gamma_i]-[\gamma]$ where $\gamma$ is a multicurve on $\Sigma$ and $\#$ denotes the connected sum. Expand $[\gamma\#\gamma_i]-[\gamma]=f+\epsilon f'$ so that we need to prove the equation $f'=-2 L_{X_f}T/T$. 

 \begin{figure}[htbp]
\centering
  \def\svgwidth{8cm}
 \executeiffilenewer{handlebody.svg}{handlebody.pdf}%
 {inkscape -z -D --file=handlebody.svg %
 --export-pdf=handlebody.pdf --export-latex}%
 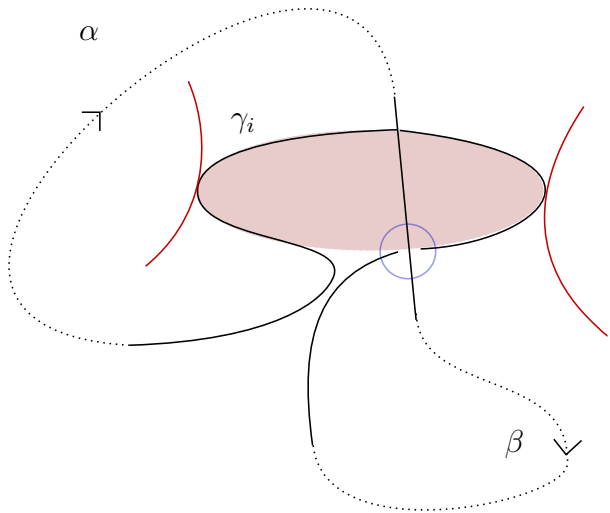%

  \caption{Connected sum}
  \label{handlebody} 
\end{figure}
Let us first compute in detail $f'$. Let $\rho:\pi_1(M)\to \sldeux(\C)$ be a representation, and consider the evaluation of $f'$ at $[\rho]$: it is a sum over all crossings between $\gamma$ and $\gamma_i$ as pictured in Figure \ref{handlebody}. We denote by $\alpha$ the part of $\gamma$ from the base point to the intersection point, and by $\beta$ the remaining part and put $A=\rho(\alpha)$ and $B=\rho(\beta)$ so that $\chi_\rho(\gamma)=-\tr(BA)$. The contribution to $f'$ of the intersection point in consideration is obtained by taking the difference of the two resolutions of the blue crossing, that is $\tr A \tr B +\tr B^{-1}A=\tr AB +2\tr AB^{-1}$. Notice that we would have gotten the opposite if $\gamma$ had crossed $\gamma_i$ in the opposite way. 

Let us now compute the divergence of the vector field $X_f$. Suppose that $\rho$ is irreducible. It is well-known (see for instance \cite{goldman}) that the tangent space at $[\rho]$ of the character variety $X(\Sigma)$ is the space $H^1(\Sigma,\ad_\rho)$. Let $f_\gamma: X(\Sigma)\to \C$ be the trace function defined by $f_\gamma([\rho])=\tr \rho(\gamma)$. Its derivative at $[\rho]$ is given by the twisted 1-cycle $\gamma\otimes \rho(\gamma)_0$ where for any $A\in M_2(\C)$ we have set $A_0=A-\frac{\tr(A)}{2}\rm{Id}$. From the expression $f=-f_{\gamma\#\gamma_i}+f_{\gamma}$, we have $df([\rho])=-\gamma\#\gamma_i\otimes \rho(\gamma\#\gamma_i)_0+\gamma\otimes \rho(\gamma)_0$. If $\rho$ extends to $\pi_1(M)$, then $\rho(\gamma_i)=\rm{Id}$. Decomposing the cycle $\gamma\#\gamma_i$, we get finally $df([\rho])=-\gamma_i\otimes\rho(\gamma)_0$. 

The vector field $X_f$ corresponding to this differential form on $X(M)$ can be interpreted similarly to the Goldman twist flows as follows. Pick a cellular decomposition of $\Sigma$ containing $\gamma_i$ and a 1-cocyle on this decomposition representing $\rho$. Then thicken $\gamma_i$ and insert a small transversal edge denoted by $e$. We define a 1-cocycle $\rho_t$ by taking the same cocycle than $\rho$ but setting $\rho_t(e)=\exp(t\rho(\gamma)_0)$. Then the derivative $X_f=\frac{d}{dt}\big | _{t=0}\,\rho_t$ is the symplectic dual to $\gamma_i\otimes \rho(\gamma)_0$. 

Recall that $\pi_1(H)$ is a free group, say of rank $g$, so that its character variety is $X(H)=\sldeux(\C)^g/\!\!/\sldeux(\C)$. The Reidemeister torsion on $X(H)$ is the push-forward of the  Haar measure $\mu$ on $\sldeux(\C)^g$ with a normalization which does not matter in this proof.

Let $G$ be a Lie group with Lie algebra $\mathcal{G}$ and left-invariant measure $\mu$. In the left trivialization, a vector field $X$ on $G$ may be viewed as a map $F:G\to \mathcal{G}$: in formula,  $X(g)=\frac{d}{dt}\big |_{t=0}\, e^{tF(g)}g$ . As $\mu$ is invariant by left translations, we get for any function $f$ and left invariant vector field $\xi$ the equation $L_{f\xi}\mu=df\wedge i_\xi\mu+fL_\xi\mu= df\wedge i_\xi\mu$ from which we get $\diverg_{\mu}(X)=\tr(DF)$ by decomposing $F$ on a basis of $\mathcal{G}$. In this last formula, we used again the left trivialization to see $DF$ as an endomorphism of $\mathcal{G}$. 

Let us apply this fact to compute the divergence $\diverg_\mu(X_f)$: write $\gamma$ as a word $w$ in the generators $t_1,\ldots,t_g$ of $\pi_1(H)$. The vector field $X_f$ corresponds the map $F:\sldeux(\C)^g\to\liedeux(\C)^g$ given by $F(A_1,\ldots,A_g)=(w(A_1,\ldots,A_g)_0,0,\ldots,0)$. Hence $\diverg_{\mu}(X_f)$ is the divergence of the map $F_1:A_1\mapsto w(A_1,\ldots,A_g)_0$. The derivative of $F_1$ is a sum over the occurrences of $t_1$ in $w$, one for each intersection of $\gamma$ and $\gamma_i$. 

For an occurrence of the form $UA_1V$, the contribution to the derivative is $\xi\mapsto (U\xi A_1V)_0$ and its trace is $\frac{1}{2}\tr(UA_1V)+\tr(UV^{-1}A_1^{-1})$. Writing $A=UA_1$ and $B=V$, we get the same as half the contribution of $f'$. For an occurence of the form $UA_1^{-1}V$, the contribution of the derivative is $\xi\mapsto -(UA_1^{-1}\xi V)_0$ and its trace is $-\frac{1}{2}\tr(UA_1^{-1}V)-\tr(UA_1^{-1}V^{-1})$. Writing $A=UA_1^{-1}$ and $B=V$, we get opposite contribution. Observing that this last case corresponds to the one pictured in Figure \ref{handlebody}, we finally get the claimed formula.

\end{proof}

\section{Twisted self-linking}\label{selflink}

\subsection{Definition of the twisted self-linking}
Let $M$ be a connected 3-manifold and $\rho:\pi_1(M)\to \sldeux(\C)$ be a representation.  We denote by $\ad_\rho$ the adjoint representation on $\liedeux(\C)$ and by $\bil_\rho$ the space of bilinear forms on $\liedeux(\C)$ viewed as a $\pi_1(M)$-bi-module via the action 
$$(\gamma,\delta).b(\xi,\eta)=b(\rho(\gamma)^{-1}\xi\rho(\gamma),\rho(\delta)^{-1}\eta\rho(\delta)).$$
We may think of $\ad_\rho$ and $\bil_\rho$ as flat bundles over $M$ and $M\times M$ respectively. The fibres of $\ad_\rho$ should be (non-canonically) identified with $\liedeux(\C)$ whereas the holonomy of the same bundle between two points may be (non-canonically) identified with an element of $\sldeux(\C)$.

If $X$ is a topological space equipped with a flat bundle $E\to X$, whenever we have a triple $(N,f,s)$ where $N$ is a $k$-manifold with boundary, $f:N\to X$ is a continuous map and $s$ is a flat section of $f^*E$, it defines a twisted chain in $C_k(X,E)$. It is a cycle if $N$ is closed, we will use freely this way of constructing twisted cycles as in the following example.

Given a map $\gamma:S^1\to M$  we define a cycle $z_\gamma$ in $Z_1(M,\ad_\rho^*)$ as the triple $(S^1,\gamma,s)$ where $s(x)$ is the  linear form on the fiber $(\ad_\rho)_{\gamma(x)}$ which maps $\xi$ to $\tr(\xi A)$ where $A$ is the holonomy of $\ad_\rho$ along $\gamma$ starting from $\gamma(x)$. This cycle represents geometrically the first derivative of the trace function $f_\gamma$: we wrote it more concisely as $\gamma\otimes \rho(\gamma)_0$ in the proof of Proposition \ref{handle}. Notice that we identify $\liedeux(\C)$ with its dual by using the trace. 

We construct below a twisted 2-chain representing the second derivative of $f_\gamma$: as in differential geometry the Hessian is a well-defined quadratic form only if the first derivative vanishes, the 2-chain we construct will not be closed in general. 

We define $C_2 S^1$ as the product $S^1\times S^1$ blown-up along the diagonal. That is, a point $(x,y)\in C_2 S^1$ is either a pair with $x\ne y$, a pair of the form $(x_-,x_+)$ or a pair of the form $(x_+,x_-)$ where $x_-$ lies before $x_+$ on the circle. Observe that $C_2S^1$ is an annulus whose boundary is made of two copies of the diagonal $S^1\subset S^1\times S^1$. 

Let $\gamma:S^1\times [-1,1]\to M$ be an embedding. We set $\gamma^+(x)=\gamma(x,1), \gamma^-(x)=\gamma(x,-1)$ and $\gamma^0(x)=\gamma(x,0)$. Let us define the 2-chain $z^2_\gamma\in C_2(M^2\setminus\Delta,\bil_\rho)$
as the triple $(C_2S^1,\gamma^-\times\gamma^+,s)$ where for $x\ne y$, $s(x,y)$ is the following 
bilinear form on $(\ad_\rho)_{\gamma^-(x)}\times (\ad_\rho)_{\gamma^+(y)}$:

\[(\xi,\eta)\mapsto \tr(\xi A \eta B).\]
In this formula, $A$ (resp. $B$) is the holonomy of $\rho$ along $\gamma$ from $\gamma^0(x)$ to $\gamma^0(y)$ (resp. from $\gamma^0(y)$ to $\gamma^0(x)$), noticing that we identified the fibers of $\ad_\rho$ along the path $t\mapsto \gamma(x,t)$.

This cycle is not closed and we find that $\partial z^2_\gamma$ is the difference $(S^1,\beta^+,s^+)-(S^1,\beta^-,s^-)$ where $\beta^+(x)=(\gamma^+(x),\gamma^-(x))$ and $s^+(x):\xi\otimes \eta\mapsto \tr(\xi\eta A_x)$ while $\beta^-(x)=(\gamma^-(x),\gamma^+(x))$ and $s^-(x):\xi\otimes \eta\mapsto \tr(\eta\xi A_x)$.

Let $U$ be a tubular neighborhood of the diagonal $\Delta$ in $M^2$ such that for all $x\in S^1$ and $s,t\in [-1,1]$, $(\gamma(x,s),\gamma(x,t))\in U$. The 
two maps $(\gamma^+,\gamma^-)$ and $(\gamma^-,\gamma^+)$ are homotopic in $U$ so that $\partial z^2_\gamma$ is homologous to the cycle $(S^1,\gamma^-\times\gamma^+,\xi\otimes \eta\mapsto \tr([\xi,\eta]A_x)$ in the same notation as before.

The above bilinear form is antisymmetric: we denote by $\Lambda^2\ad_\rho^*$ this coefficient system over $U$. We summarize the preceding discussion by writing $\partial z^2_\gamma\in Z_1(U\setminus \Delta,\Lambda^2 \ad_\rho^*)$. In order to make $z^2_\gamma$ into a cycle, we form the following complex:
\[\tilde{C}_*(M,\rho)=C_*(M^2\setminus \Delta,\bil_\rho)/C_*(U\setminus\Delta,\Lambda^2\ad_\rho^*).\]

We will denote by $\mathcal{H}(M,\rho)$ the space $H_2(\tilde{C}_*(M,\rho))$. The 2-cycle $z^2_\gamma\in \mathcal{H}(M,\rho)$ will be interpreted at the same time as the second derivative of $f_\gamma$ at $\rho$ and the self-linking of $\gamma$ with coefficients in $\ad_\rho$. 

\subsection{Geometric interpretations of the self-linking}
The self-linking interpretation should be clear from the following remark: if $L_1$ and $L_2$ are two disjoint knots in $\R^3$, the product $L_1\times L_2$ makes sense as a 2-cycle in $H_2((\R^3)^2\setminus \Delta,\Z)$. Let $u$ be the 2-cycle given by $\{0\}\times S^2$: it is a generator of $H_2((\R^3)^2\setminus \Delta,\Z)$ and one has $L_1\times L_2=\lk(L_1,L_2)u$. This comes from the interpretation of the Gauss formula as the degree of the map from $L_1\times L_2$ to $S^2$ mapping $(x_1,x_2)$ to $\frac{x_2-x_1}{|x_2-x_1|}$. If $L_1$ and $L_2$ are the two boundary components of a banded knot $L$, we get the self-linking number of $L$. The space $\mathcal{H}(M,\rho)$ is reminiscent of this construction if we replace $\R^3$ with a 3-manifold $M$ and twist the coefficients by $\rho$. 

Let us interpret $z^2_\gamma$ as the second derivative of $f_\gamma$: suppose that $\rho:\pi_1(M)\to\liedeux(\C)$ is obtained as the holonomy of a 1-form $a\in\Omega^1(M,\liedeux(\C))$. Explicitly, the holonomy of $a$ along a path $\gamma:[0,1]\to M$ may be computed with the following formula introduced by Picard in 1890 as a method for approximating the solutions of a linear differential equation with non constant coefficients: 
\[ \tr \hol_\gamma a=\sum_{n\ge 0}\int_{0<t_1<\cdots <t_n<1}\tr(\gamma^*a(t_1)\cdots\gamma^*a(t_n)).\]

This gives $D_a \tr \hol_\gamma (b)=\int_\gamma \tr b \hol_\gamma(a)=\langle z_\gamma,b\rangle$. In the same way we have

\[\frac{1}{2}D^2_a \tr \hol_\gamma(b_1,b_2)=\int_{S^1\times S^1}\tr(b_1\hol_{x\to y} (\gamma^*a) b_2\hol_{y\to x}(\gamma^*a))=\langle z^2_\gamma, b_1\times b_2\rangle.\]
This formula was indeed the starting point for defining the cycle $z^2_\gamma$. The fact that this second derivative should play a role in the derived skein module comes from the path integral as explained in Subsection \ref{witten}.

\subsection{Relation with the derived skein module}
Given an embedding $i:D^3\to M$ and $\phi$ any bilinear map on the fiber of $(\ad_\rho)_{i(0)}$ - that we can view as a flat section of $i^*\bil_\rho$ - we set $u\otimes \phi=(S^2,i|_{S^2},\phi)\in \mathcal{H}(M,\rho)$. This class only depends on $\phi$. 
We will need the following properties of elements in $\mathcal{H}(M,\rho)$. 

\begin{proposition}\label{derivproduit}
\begin{enumerate}
\item
Let $F:S^1\times[-1,1]\times[0,1]\to M$ be a 1-parameter family of embeddings where we write $F_s(x,t)=F(x,t,s)$.

The map $F$ can be viewed as a chain in $C_3(M^2,\bil_\rho)/C_*(U\setminus\Delta,\Lambda^2\ad_\rho^*)$ such that $\partial F = z^2_{F_1}-z^2_{F_0}$. It will be chain in $\tilde{C}_3(M,\rho)$ if moreover one has $F(x,1,s)\ne F(y,-1,t)$ for all $x,y\in S^1$ and $s,t\in [-1,1]$.  
\item
Let $\alpha,\beta:S^1\times[0,1]\to M$ be two embeddings satisfying $\alpha(x,t)=\beta(y,s)\iff x=y=1$ and $s=t$. Then one can define the composition $\alpha\beta:S^1\times[0,1]\to M$ and the following identity holds in $\mathcal{H}(M,\rho)$: 

\[z^2_{\alpha\beta}+z^2_{\alpha\beta^{-1}}=\tr \rho(\beta) z^2_\alpha+z_\alpha^-\times z_\beta^+ +z_\beta^-\times z_\alpha^++\tr \rho(\alpha)z^2_\beta.\]

\item
Observe that there is a unique $\sldeux(\C)$-equivariant projection $\pi$ from $\bil_\rho$ to the space of symmetric invariant bilinear forms (or Killing forms). If $\rho$ is Zariski-dense, then $u\otimes \phi=u\otimes \pi(\phi)$ in $\mathcal{H}(M,\rho)$. 
\end{enumerate}
\end{proposition}
\begin{proof}
The first point is clear as one can repeat the construction of the the chain $z^2_\gamma$ by adding the parameter $s$. The condition of non-intersection ensures that this cycle takes values in $M^2\setminus\Delta$, proving the first part. 

The second formula is formally the second derivative of the trace identity $f_{\alpha\beta}+f_{\alpha\beta^{-1}}=f_\alpha f_\beta$. 
We first observe that the composition $\alpha\beta$ is no longer an embedding. It does not prevent from defining $z^2_{\alpha\beta}$ as one still has $\alpha\beta(x,1)\ne\alpha\beta(y,-1)$ for all $x,y\in S^1$. The second observation is that if we reverse the orientation of a component, then the cycle get multiplied by $-1$. As a convention, we will associate an element $\xi\in (\ad_\rho)_{\gamma(t)}$
with a local orientation of $\gamma$ at $t$. Reversing this orientation amounts in changing $\xi$ with $-\xi$. This will allow us to glue knots with incompatible orientations. 

Cut $S^1$ into the two halves $S^1_{\pm}=\{z\in S^1, \pm \operatorname{Im}(z)\ge 0\}$: then we have the decomposition $C_2S^1=C_2S^1_+\cup C_2S^1_-\cup S^1_+\times S^1_-\cup S^1_-\times S^1_+$ to which correspond a linear decomposition of $z_{\alpha\beta}$ and $z_{\alpha\beta^{-1}}$. We check the identity of the lemma by looking at the restriction to each piece. 

 \begin{figure}[htbp]
\centering
  \def\svgwidth{13cm}
 \executeiffilenewer{leibnitz.svg}{leibnitz.pdf}%
 {inkscape -z -D --file=leibnitz.svg %
 --export-pdf=leibnitz.pdf --export-latex}%
 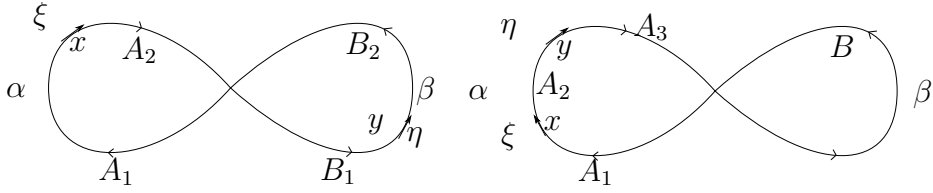%

  \caption{Composition formula}
  \label{leibnitz} 
\end{figure}
With the notation of the left part of Figure \ref{leibnitz}, if $x\in S^1_+$ and $y\in S^1_-$, the bilinear form associated to $z^2_{\alpha\beta}$ is $\tr( B_2\eta B_1A_2\xi A_1)$ while the one associated to $z^2_{\alpha\beta^{-1}}$ gives $-\tr(B_1^{-1}\eta B_2^{-1}A_2\xi A_1)$. Introduce the notation $\begin{pmatrix} a & b\\c & d \end{pmatrix}^*=\begin{pmatrix} d & -b \\ -c & a\end{pmatrix}$ so that we have $\xi^*=-\xi$ for $\xi\in \liedeux(\C)$ and $A^*=A^{-1}$ for $A\in\sldeux(\C)$. The contribution of $z^2_{\alpha\beta}+z^2_{\alpha\beta^{-1}}$ is 
$\tr(BA)+\tr(B^*A)$ where $A=A_2\xi A_1$ and $B=B_2\eta B_1$. From the identity $B+B^*=\tr(B)\id$, we get $\tr(A)\tr(B)$ and recognize the bilinear form associated to the product $z_\alpha^-\times z_\beta^+$. We obtain in the same way the term $z_\beta^-\times z_\alpha^+$. 

If $x,y\in S^1_+$ are as in the right side of Figure \ref{leibnitz}, the contribution of $z^2_{\alpha\beta}+z^2_{\alpha\beta^{-1}}$ is $\tr(BA_3\eta A_2\xi A_1)+\tr(B^{-1}A_3\eta A_2\xi A_1)=\tr(B)\tr(A_3\eta A_2\xi A_1)$. This corresponds to $\tr \rho(\beta) z^2_\alpha$ as expected, the last term is treated in the same way.

For the third point, we observe that it is sufficient to show that if $\pi(\phi)=0$ in $H_2(M^2\setminus \Delta,\bil_\rho)$ then $u\otimes \phi=0$. Replace for this proof $M$ by its interior: the exact sequence of the pair $(M^2,M^2\setminus\Delta)$ and the excision property give 
\[H_3(U,\partial U;\bil_\rho)\to H_2(M^2\setminus\Delta,\bil_\rho)\to H_2(M^2,\bil_\rho)\to H_2(U,\partial U,\bil_\rho).\] 
As $M$ has a trivial tangent bundle, $U$ is diffeomorphic to $M\times D^3$ and hence by K\"unneth formula : $H_*(U,\partial U,\bil_\rho)=H_*(M,\bil_\rho)\otimes H_*(D^3,\partial D^3)$. 
This implies that $H_2(U,\partial U,\bil_\rho)=0$ and $H_3(U,\partial U,\bil_\rho)=H_0(M,\ad_\rho^*\otimes\ad_\rho^*)$. Using Clebsch-Gordan rule, we have the following decomposition $\ad_\rho^*\otimes \ad_\rho^*\simeq\C\oplus S^2V_\rho\oplus S^4V_\rho$ where $V_\rho$ denotes the space $\C^2$ viewed as a $\pi_1(M)$-module. We observe that $\pi$ corresponds to the projection on the first factor. By Zariski density, all representations $S^n V_\rho$ are irreducible, hence we get $H_0(M,\ad_\rho^*\otimes\ad_\rho^*)=H_0(M,\C)=\C$. By assumption, $\pi(\phi)=0$ and hence its projection on $H_0(M,\C)$ vanishes, which shows the lemma.

\end{proof}

%
%
We will normalize the Killing form by the formula $K(\xi,\eta)=\frac{1}{6}\tr(\xi\eta)$. The projection $\pi$ can be computed by the formula $\pi(\phi)=Q(\phi)K$ where $Q(\phi)=\phi(H,H)+2\phi(E,F)+2\phi(F,E)$ in the standard basis $H,E,F$ of $\liedeux(\C)$.  

We set $\mathcal{H}_\epsilon(M,\rho)=\C\oplus \epsilon\mathcal{H}(M,\rho)$ and endow it with a $\C[\epsilon]/(\epsilon^2)$-module structure by setting $\epsilon 1=u\otimes K$. 
\begin{proposition}
 Let $M$ be a 3-manifold and $\rho:\Gamma\to \sldeux(\C)$ be a Zariski-dense representation. The following map $\Phi:S'(M)\to \mathcal{H}_\epsilon(M,\rho)$ is well-defined:
$$\Phi([L])=\chi_\rho(L)+\epsilon (-1)^{\#L}z(L)$$
where $L=\coprod_{i=1}^n L_i$ and 
$$z(L)=\sum_i z^2(L_i) \prod_{j\ne i} \tr \rho(L_j)+\sum_{i\ne j} z^-(L_i)\times z^+(L_j)\prod_{k\notin\{i,j\}}\tr \rho(L_k)$$
\end{proposition}
\begin{proof}
We extend $\Phi$ to all elements in $\mathcal{S}'(M)$ by $\C[\epsilon]/(\epsilon^2)$-linearity.
Let $L$ be a banded link and $U$ be an unknot disjoint from $L$. The Kauffman relation implies $\Phi([L\cup U])=-2\Phi([L])$ as $-t^2-t^{-2}=-2 \mod \epsilon^2$ hence we need to show the identity $z(L\cup U)=2z(L)$. This comes directly from the fact that $\tr \rho(U)=2$ and $z^2_U=z^\pm_U=0$ in $\mathcal{H}(M,\rho)$.

Let $L_\times, L_0,L_\infty$ be three links which form a Kauffman relation $[L_\times]=t[L_0]+t^{-1}[L_\infty]$. As in Equation \eqref{phi}, we have to show the relation: 
\[(-1)^{\# L_\times}z(L_\times)+(-1)^{\#L_0}z(L_0)+(-1)^{\#L_\infty}z(L_\infty)=\chi_\rho(L_0)-\chi_\rho(L_{\infty}).\]

We have two cases to consider: in the first case, the strands are connected as in the picture which represents simultaneously the three links $L_\times,L_0$ and $L_\infty$. 
\begin{figure}[htbp]
\begin{center}
  \def\svgwidth{12cm}
 \executeiffilenewer{K3.svg}{K3.pdf}%
 {inkscape -z -D --file=K3.svg %
 --export-pdf=K3.pdf --export-latex}%
 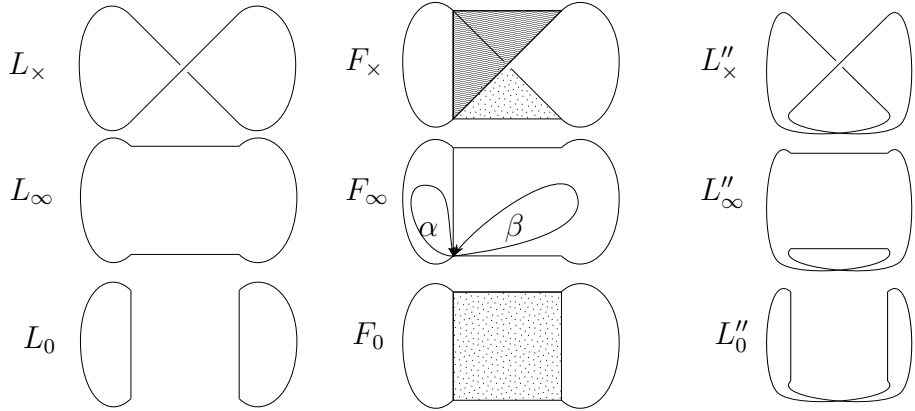%

\caption{Cobordisms for Kauffman relations}\label{cobo}
\end{center}
\end{figure}

Set $X=(-1)^{\# L_\times}z(L_\times)+(-1)^{\#L_0}z(L_0)+(-1)^{\#L_\infty}z(L_\infty)$ and let $F^\times,F^0$ and $F^\infty$ be the deformations of $L_\times,L_0$ and $L_\infty$ shown in Figure \ref{cobo}. As stated in the first point of Proposition \ref{derivproduit}, to each deformation $F$ from $L$ to $L'$ correspond a 3-chain denoted by $z^2(F)$ so that $\partial z^2(F)=z^2(L')-z^2(L)$. The new links $L_\times',L_0'$ and $L_\infty'$ are just as in the second point of Proposition \ref{derivproduit}. Hence, checking the signs, we get $(-1)^{\# L_\times}z(L_\times')+(-1)^{\#L_0}z(L_0')+(-1)^{\#L_\infty}z(L_\infty')=0$ so that we have $-X=\partial z^2(F_\times)+\partial z^2(F_0)+\partial z^2(F_\infty)$. If the deformation $F$ satisfied the equation $F(x,1,s)\ne F(y,1,t)$ for all $x,y\in S^1$ and $t,s\in [0,1]$, then $z_2(F)$ would be a chain in $C_3(M^2\setminus\Delta,\bil_\rho)$ and hence we would have $z^2(L)=z^2(L')$ in $\mathcal{H}(M,\rho)$. This is the case of $F_0$ and $F_\infty$. We observe that for $F_\times$, this condition fails for one precise set of parameters. By cutting out a small ball around that point, we see that $z^2(L_\times)-z^2(L'_\times)=u\otimes \phi$ where $\phi(\xi,\eta)=\tr(B\eta A\xi)$ with $A=\rho(\alpha)$ and $B=\rho(\beta)$. Hence $X=Q(\phi)=\tr(A)\tr(B)+\tr(AB^{-1})=\chi_\rho(L_0)-\chi_\rho(L_\infty)$. 

In the second case, the links are presented on the right hand side of the figure as $L''_\times,L''_\infty$ and $L''_0$. We need to show the identity
$z(L''_0)+z(L''_\infty)-z(L''_\times)-\chi_\rho(L''_0)+\chi_\rho(L''_\infty)=0$. Using the same surfaces and the same argument, we find that $z(L''_0)+z(L''_\infty)-z(L''_\times)= -Q(\phi)
$ where $\phi(\xi,\eta)=\tr(A\xi)\tr(B^{-1}\eta)$. A computation gives $Q(\phi)=\tr(AB^{-1})-\tr(AB)=-\chi_\rho(L''_\infty)+\chi_\rho(L''_0)$. This ends the proof of the proposition.
\end{proof}

\subsection{Properties of the space $\mathcal{H}(M,\rho)$}\label{properties}

Let $M$ be any 3-manifold with $\partial M=\coprod_i \Sigma_i$. We suppose that no component $\Sigma_i$ is a sphere and set $d=\sum_i\max(1,3 g(\Sigma_i)-3)$. This number is the minimal dimension of the character variety $X(M)$ (and the most regular case).
Let $\rho:\pi_1(M)\to \sldeux(\C)$ such that 
\begin{enumerate}
\item $\rho$ has Zariski-dense image.
\item $H^1(M,\ad_\rho)$ has dimension $d$. 
\item the restriction of $\rho$ to toric boundary components is diagonalizable and non-central.
\end{enumerate}

Observe that the $\pi_1(M)$-module $\Lambda^2\ad_\rho^*$ is isomorphic to $\ad_\rho$ via the map $X\mapsto ( (\xi,\eta)\mapsto \tr([\xi,\eta]X))$. 
Hence, we have the following isomorphisms - using again the fact that the tangent bundle of a 3-manifold is always trivializable. 
\[H_*(U\setminus\Delta,\Lambda^2\ad_\rho^*)\!=\!H_*(U\setminus \Delta,\ad_\rho)\!=\!H_*(M\times S^2,\ad_\rho)\!=\!H_*(M,\ad_\rho)\otimes H_*(S^2).\]

By excision, we have $H_*(M^2,M^2\setminus\Delta,\bil_\rho)=H_{*-3}(M,\ad_\rho^*\otimes \ad_\rho^*)$ and the exact sequence of the pair $(M^2,M^2\setminus\Delta)$ gives the following sequence where the space $\C$ is generated by $u\otimes K$.  
\[H_3(M^2,\bil_\rho)\overset{I}{\to}\C\to H_2(M^2\setminus\Delta,\bil_\rho)\to H_2(M^2,\bil_\rho)\to 0.\]
 The map $I$ has a very natural interpretation: by K\"unneth formula and identifying $\ad_\rho$ with $\ad_\rho^*$ using the trace, we have $H_3(M^2,\bil_\rho)=H_1(M,\ad_\rho)\otimes H_2(M,\ad_\rho)$ $\oplus H_2(M,\ad_\rho)\otimes H_1(M,\ad_\rho)$ and the map $I$ is simply the intersection followed by the trace. 
With the assumption on the dimension of $H^1(M,\ad_\rho)$, we know from the exact sequence of the pair $(M,\partial M)$ that the map $H_1(\partial M,\ad_\rho)\to H_1(M,\ad_\rho)$ is surjective (see for instance \cite{numazu} Proposition 4.4), hence the intersection form $I$ vanishes.

The construction of $\tilde{C}_*(M,\rho)$ as a quotient of two complexes and the previous computations give the following long exact sequence:

$$\xymatrix{
H_2(M,\ad_\rho)\ar[r]^\alpha& H_2(M^2\setminus\Delta,\bil_\rho)\ar[r]& \mathcal{H}(M,\rho)\to H_1(M,\ad_\rho)\to 0}$$
Let us prove that the map $\alpha$ vanishes under our hypothesis. As above, from the exact sequence of the pair $(M,\partial M)$ we get that the map $H_2(\partial M,\ad_\rho)\to H_2(M,\ad_\rho)$ is an isomorphism. In particular each toric component of $\partial M$ gives a generator of $H_2(M,\ad_\rho)$. 
Let $T$ be a torus component of $\partial M$ and $T'$ a parallel copy of $T$ pushed into $M$. Our assumptions on the restriction of $\rho$ to $T$ imply $H_0(T,\ad_\rho)\simeq H_2(T,\ad_\rho)=\C$ and $H_1(T,\ad_\rho)$ has dimension 2. We denote by $H$ the usual diagonal matrix with entries $1$ and $-1$.

The generator $[T]\otimes H$ of $H_2(M,\ad_\rho)$ is mapped by $\alpha$ to the cycle $(T,i,s)$ where $i(x)=(x,x')$ with $x\in T$ and $x'$ parallel to $x$ in $T'$ and $s(x):\xi\otimes \eta\mapsto \tr([\xi,\eta]H)=\tr(\xi [H,\eta])$. In K\"unneth isomorphism, this cycle is written $\sum_k \sum_i a^k_i\otimes \ad_H(a^k_i)^\#$ where $(a^k_i)_i$ is a basis of $H_k(T,\ad_\rho)$ and $(a^k_i)^\#$ is its Poincar\'e dual basis. But all basis vectors in $H_*(T,\ad_\rho)$ can be chosen proportional to $H$, and hence will be killed by $\ad_H$, this proves that $\alpha$ vanishes. We observe that the map $\alpha$ is dual to the cup-bracket operation which vanishes in our context.

\begin{remark}
From these considerations, the space $\mathcal{H}(M,\rho)$ looks like the target space for a 2-jet over the character variety. This is expected from the interpretation of $z^2_\gamma$ as a second derivative, however, we could not give a precise meaning to this observation. 
\end{remark}

\subsection{Applications}
In the case when $M$ is closed the situation is much simpler. Suppose that $\rho:\pi_1(M)\to \sldeux(\C)$ is a Zariski-dense representation satisfying $H^1(M,\ad_\rho)=0$, then $u\otimes K$ is a generator of $\mathcal{H}(M,\rho)$ and hence one can define a $\C[\epsilon]/(\epsilon^2)$-linear map $\mathcal{S}'(M)\to \C[\epsilon]/(\epsilon^2)$ extending the character $\chi_\rho$. The existence of this maps proves the regularity of the skein module $\mathcal{S}'(M)$ (at least locally at $\rho$) but more importantly, it gives the geometric interpretation of the first-order term in Witten's expansion formula as explained in Subsection \ref{witten}.

Now, let $M$ be a 3-manifold with non-empty boundary and let $f+\epsilon f'$ be an element of $\mathcal{P}'$. This implies that for any $\rho:\pi_1(M)\to \sldeux(\C)$ satisfying the assumptions of Proposition \ref{properties}, $f'+\Phi(f)=0$ in $\mathcal{H}(M,\rho)$. The image of $\Phi(f)$ in $H_1(M,\ad_\rho)$ vanishes because it corresponds to $df$ and $f$ belongs to $\mathcal{P}$ and hence vanishes on $X(M)$. We would answer Question \ref{conjtransport} if we were able to show the formula $\Phi(f)=\frac{1}{2}\diverg_T(X_f)$.

\end{document}